\documentclass[a4paper,12pt]{amsart}
\usepackage[utf8]{inputenc}
\usepackage[english]{babel}
\usepackage{a4wide,soul}
\usepackage{tikz-cd}
\usepackage{amsthm, amssymb, amsmath}
\usepackage{enumitem}
\usepackage{xcolor}

\usepackage[cmtip,arrow]{xy}
\usepackage{pb-diagram, pb-xy}

\usepackage{etoolbox}
\apptocmd{\sloppy}{\hbadness 10000\relax}{}{}

\newtheorem{theorem}{Theorem}[section]
\newtheorem{lemma}[theorem]{Lemma}
\newtheorem{proposition}[theorem]{Proposition}
\newtheorem{corollary}[theorem]{Corollary}

\theoremstyle{definition}
\newtheorem{definition}[theorem]{Definition}
\newtheorem{example}[theorem]{Example}
\theoremstyle{remark}
\newtheorem{remark}[theorem]{Remark}
\newtheorem{fact}[theorem]{Fact}

\newcommand{\CC}{\mathbb{C}}
\newcommand{\NN}{\mathbb{N}}
\newcommand{\ZZ}{\mathbb{Z}}
\newcommand{\RR}{\mathbb{R}}

\newcommand{\A}{{\mathbb A}}

\newcommand{\C}{{\mathbb C}}
\newcommand{\N}{{\mathbb N}}

\newcommand{\Z}{{\mathbb Z}}

\renewcommand{\O}{{\mathcal O}}

\newcommand{\aut}{\mathop{\mathrm{Aut}}}

\newcommand{\Olo}{\mathcal{O}}

\newcommand{\rank}{\mathop{\mathrm{rank}}}
\newcommand{\saut}{\mathop{\mathrm{SAut}}}
\newcommand{\Aut}{\mathop{\mathrm{Aut}}}

\newcommand\dd[1]{\frac{\partial}{\partial #1}}
\newcommand{\tensor}{\otimes}

\newcommand{\Sc}{{\mathcal S}}
\newcommand{\Ks}{{\mathcal K}}

\newcommand{\sP}{{\mathcal P}}
\newcommand{\sL}{{\mathcal L}}
\newcommand{\sV}{{\mathcal V}}
\newcommand{\sH}{{\mathcal H}}
\newcommand{\sE}{{\mathcal E}}

\setstcolor{red}
\makeatletter
\@namedef{subjclassname@2020}{%
  \textup{2020} Mathematics Subject Classification}
\makeatother

\begin{document}
\author{Riccardo~Ugolini}
\address{Riccardo~Ugolini \\ Faculty of Mathematics \\ Ruhr University Bochum \\ Germany}
\email{riccardo.ugolini@ruhr-uni-bochum.de}

\author{J\"org Winkelmann}
\address{J\"org Winkelmann\\ Faculty of Mathematics \\ Ruhr University Bochum \\ Germany}
\email{joerg.winkelmann@rub.de}



\title{The density property for vector bundles}

\begin{abstract}
We prove that holomorphic vector bundles over Stein manifolds with the density property also satisfy the density property, provided that the total space is holomorphically flexible. We apply this result to provide a new class of Stein manifolds with the density property.
\end{abstract}

\maketitle
\tableofcontents

\section{Introduction}

The class of Stein manifolds with the density property is at the center of the Andersén-Lempert Theory. The latter began when Andersén and Lempert first described the group of holomorphic automorphisms of complex affine spaces $\C^n, \ n>1$ \cite{AndersenLempert1992}. Later, Varolin generalized their ideas and introduced the density property:

\begin{definition}[\cite{Varolin2001}]
A complex manifold $X$ has the \textit{density property} if the Lie algebra generated by {complete} holomorphic vector fields (those whose flows are defined for all complex times) is dense in the Lie algebra of all holomorphic vector fields in the compact-open topology.
\end{definition}

Note that the flow of a complete vector field at a fixed time $t \in \C$ defines an automorphism of $X$. Under the additional condition that $X$ is Stein, the density property hence guarantees the existence of a large number of automorphisms. This has many striking consequences; we refer the interested reader to the monograph \cite{ForstnericBook} or the more recent survey \cite{ALSurvey}.

Our main result is the following:
\begin{theorem}[Main theorem] \label{main}
  Let  $X$ be a Stein manifold with density property, $\pi:E\to X$ a holomorphic
  vector bundle.

  Assume that there exists  an automorphism
  $\phi$ of the total space $E$ and points $p,q\in E$ such that
  $\pi(p)=\pi(q)$, but $\pi(\phi(p))\ne\pi(\phi(q))$.

 Then $E$ satisfies the density property.
\end{theorem}

For the proof, we distinguish two cases: $\rank(E)\ge 2$ is
  dealt with in Theorem~\ref{mainrank2} and Theorem
  \ref{mainrank1} takes care of the case $\rank(E)=1$.
  
The first application of Theorem \ref{main} concerns a specific class of {\em flexible} manifolds.

\begin{definition}
  A complex algebraic manifold $X$ is flexible at a point $p \in X$ if the tangent space $T_p X$ is spanned by the tangent vectors to the orbits $Hp$ of
  unipotent one-parameter algebraic subgroups $H \subset \aut_{alg}(X)$ of the group of algebraic automorphisms. We say that $X$ is flexible if it is flexible at all of its points.
\end{definition}

These one-parameter subgroups induce algebraic
actions by the additive group of $\C$, i.e., they correspond to
{\em locally nilpotent derivations} (LNDs). In other words: $X$ is flexible
at $p$ if for every $v\in T_pX$
there exists an algebraic vector field $\Theta$ such that the
induced derivation $\theta$ of the function ring $\C[X]$ satisfies
\[
\forall f\in \C[X]:\exists n:
\underbrace{\theta(\ldots(\theta(}_{\text{$n$-times}}f))\ldots )=0
\]
These LNDs proved useful in many aspects of the Andersén-Lempert theory; more information can be found in the monograph \cite{LNDBook}.

Denoting by $\saut(X)$ the group generated by all unipotent
one-parameter subgroups of $\aut_{alg}(X)$,
the main result about flexible affine manifolds is the following:

\begin{theorem} \cite[Theorem 0.1]{FlexibleSurvey} \label{flexibilitythm}
Let $X$ be an affine algebraic manifold of dimension greater than one. Then, the following are equivalent:
\begin{enumerate}
\item $X$ is flexible
\item the group $\saut(X)$ acts transitively
\item the group $\saut(X)$ acts $n$-transitively for every $n \in \N$
\end{enumerate}
\end{theorem}

An action is $n$-transitive if any ordered $n$-tuple of distinct points can be sent to any other such $n$-tuple by the action.
Condition $(iii)$
implies that the automorphism group of a
flexible manifold is infinite dimensional.

An easy example of a flexible variety is $\C^ n$ for $n\ge 2$.

In view of Theorem \ref{flexibilitythm}, we have the
Corollary of our main Theorem:

\begin{corollary} \label{maincor}
  Let $X$ be a Stein manifold with density property,
   $E \to X$ a holomorphic vector bundle such that the total space $E$ is a flexible affine variety. 

Then $E$ satisfies the density property.
\end{corollary}

\begin{proof}
The automorphism group of the complex manifold $E$ acts
  $2$-transitively on $E$, because $E$ is flexible.
  Hence we may choose points $p,q\in E $ with $\pi(p)=\pi(q)$
  and an automorphism $\phi$ of $E$ such that $\pi(\phi(p))\ne
  \pi(\phi(q))$. Now we may invoke Theorem~\ref{main}.

\end{proof}

This raises the question how to check whether $E$ is flexible.
It is therefore useful to know that flexibility of the base manifold
$X$ implies the flexibility of the total space $E$.

\begin{theorem}[=Theorem~\ref{vb-flexible}] \label{flextotspace}
  Let $E\to X$ be an algebraic  vector bundle over a flexible smooth
  affine variety $X$.

  Then $E$ is a flexible affine variety.
\end{theorem}

The proof of the main theorem for vector bundles of rank greater than one
contains two main ingredients: Proposition \ref{approximable} gives that all vector fields on $E$ that are compatible with the projection map are limits of Lie combinations of complete vector fields, while Theorem \ref{no-inter} grants that there is no intermediate closed Lie algebra between the Lie algebra of vector fields compatible with the projection map and the Lie algebra of all vector fields. The fact that $E$ is flexible is used to guarantee the existence of a complete vector field which is not compatible with the projection map.

The case of line bundles is handled separately in Section \ref{LineBundle}. The main difference is that the fiber does not satisfy the density property, hence a new approach is required when approximating vertical vector fields. 

Using Theorem \ref{main}, we are able to provide a class of Stein manifolds with the density property.

\begin{corollary} \label{quasi-hom}
Let $X$ be a smooth irreducible affine variety with only constant invertible functions. Assume a reductive algebraic group $G$ acts on $X$ with an open orbit and that the unique closed orbit is biholomorphic to neither the complex plane $\CC$ nor a torus $(\CC^*)^k, \ k\in \NN$.

Then $X$ has the density property.
\end{corollary}

With the exception on the condition about the unique closed orbit, the hypotheses are the same as in \cite[Theorem 2]{GaifullinShafarevic2019} where the authors prove that such manifolds are flexible.

The fact that there is a unique closed orbit is a consequence of Luna's slice Theorem
(\cite{LunaSlice}, see also \cite[Theorem 6.7]{AlgGeoIV}).

\begin{proof}[Proof of Corollary~\ref{quasi-hom}]
Luna's slice Theorem gives that $X$ is $G$-biholomorphic to $G \times_H \CC^N$ for a reductive subgroup $H \subset G$ acting linearly on $\CC^N$ for some $N \in \NN$. In particular, the projection $G \times_H \CC^N \to G/H$ onto the first factor is a holomorphic vector bundle. Since $G$ and $H$ are reductive, $G/H$ is affine and biholomorphic to the unique closed orbit of the $G$-action on $X$. The work of Donzelli, Dvorsky, and Kaliman \cite{DonzelliDvorskyKaliman2010} guarantees that $G/H$ has the density property; the conclusion follows from Theorem \ref{main}.
\end{proof}

Manifolds equipped with a Lie group action admitting an open orbit are called \emph{quasi-homogeneous}
(or \emph{almost-homogeneous}); many quasi-homogeneous affine varieties were classified by Popov \cite{Popov1973a, Popov1973b} and Popov and Vinberg \cite{PopovVinberg1972}.

If we require $G$ to be semisimple instead of reductive, the result holds without additional hypotheses and is a corollary of the following
result:

\begin{theorem}\cite[Theorem 5.6]{FlexibleSurvey} 
  \label{semisimple}
Suppose that a connected semisimple algebraic group $G$ acts on a smooth irreducible affine variety $X$ with an open orbit.

Then $X$ is homogeneous with respect to a connected affine algebraic group $\tilde{G} \supset G$ without nontrivial characters.
\end{theorem}

For the sake of completeness, we explicitly state the result for semisimple groups; we do not claim any contribution to it.

\begin{corollary}[Arzhantsev, Flenner, Kaliman, Kutzschebauch, Zaidenberg]
  Let $X$ be an affine manifold of dimension greater than one and not isomorphic to a torus $\left(\CC^ *\right)^k,\  k \in \NN$.
  Assume that a connected semisimple algebraic group $G$ acts on $X$ algebraically with an open orbit.

Then $X$ satisfies the density property.
\end{corollary}

\begin{proof}
  By Theorem \ref{semisimple}, $X$ is a homogeneous space.
  Since a semisimple algebraic group acts with an open orbit,
  $X$ is neither one-dimensional nor a torus.
  Hence it satisfies the density property thanks to \cite[Theorem 1.3]{KalimanKutz2017}.
\end{proof}

Let us illustrate a new example of Stein manifold with the density property that can be obtained from our results:

\begin{example}
  Let $P$ be a polynomial on $\C$
  without multiple zeroes.
  Then
  \[
  X=\{(x,y,z)\in\C^3:xy=P(z)\}
  \]
  is called a {\em  Danielewski surface}. Danielewski surfaces are
  known to be flexible affine varieties and to satisfy the density
  property (\cite{KalimanKutz2008.1}). One may easily verify that
  $b_2(X)=\deg(P)-1$ for a Danielewski surface
  (compare \cite{MR1669174}).
  The classification
  of homogeneous surfaces
  (see \cite{MR638369}) implies that no simply-connected
  complex surface with $b_2(X)>1$ is biholomorphic to a quotient $G/H$ of
  complex Lie groups.
  Furthermore, $b_2(X)>0$ implies that there exist non-trivial
  holomorphic line bundles.%
  \footnote{$b_2(X)>0$ implies $H^2(X,\Z)\ne\{0\}$. On a Stein manifold
    $H^ 2(X,\Z)$ classifies holomorphic line bundles up to
    holomorphic equivalence.}
  Fix such a non-trivial line bundle over a Danielewski surface
  and call the total space $Y$.
  By Theorem \ref{flextotspace} flexibility of
  the base manifold $X$ implies that the total space $Y$ is also flexible.
  Thus $Y$ is a Stein manifold with density property
  by Corollary \ref{maincor}.
  To our knowledge, this is a new example
  of a manifold with density property
  that does not appear elsewhere in the literature. 
  \end{example}

Most of the paper is devoted to the proof of Theorem \ref{main}. Section \ref{flexibility} independently proves Theorem \ref{flextotspace}, while in the last section we discuss the existence of tame sets in the newly provided Stein manifolds with the density property; the definition of tame set and a panoramic of the context in which they arise will be provided there.

\begin{fact} \label{spanLie}
  Let $X$, $Y$ be complete vector fields on a manifold $M$.

  Then $[X,Y]$ can be approximated by
  linear combinations of complete vector fields.
\end{fact}

\begin{proof}
  For the benefit of the reader we indicate a proof although this is well
  known.
  Let $\phi_t$ be the flow associated to $X$. Then
  \[
    [X,Y]={\mathcal L}_X(Y)=\lim_{t\to\infty}\frac{\phi_t^*Y-Y}{t}
  \]
  where ${\mathcal L}$ denotes the Lie derivative.
  Both $\frac 1t\phi_t^ *Y$ and $\frac 1tY$
  are complete for every $t\in\C$.
\end{proof}

\begin{corollary}\label{closure-is-LA}
  Let $M$ be a manifold, $D$ the space of all vector fields and
  $C$ the vector space generated by all complete vector fields.

  Then the closure of $C$ in $D$ is a Lie subalgebra.
\end{corollary}

\section{Connections and horizontal vector fields}

Let $\pi:E\to X$ be a holomorphic vector bundle.

To find complete vector fields on $E$, we will lift complete vector fields on $X$ along the projection $\pi:E \to X$. In order to do so, we need a holomorphic connection.

A holomorphic connection $\nabla$
on a holomorphic vector bundle is a rule to derive
  holomorphic sections.
  It defines a map $\nabla:\Gamma(U,E)\to\Omega^ 1(U)\tensor\Gamma(U,E)$
  for open subsets $U\subset X$.
  In local coordinates  $\nabla$ evaluates to 
  \[
  \nabla\sigma\simeq d\sigma+A\sigma
  \]
  with the {``\em connection form''} $A$ being a section
  in $End(E)\tensor\Omega_X^1$.

  A connection yields a direct sum decomposition of the tangent
  bundle $T_E=H\oplus V$ where $H$ is the vector bundle of
  {\em horizontal}
  vector fields and $V$ the bundle of {\em vertical} vector fields.
  (such a direct sum decomposition is also known as
  {\em ``Ehresmann connection''}).
  A vector field $v$ on $X$ may be lifted (with respect to the given
  connection) to a ``horizontal vector field''  $\tilde v$.
  In local coordinates ($x_i$ coordinates on the base, $y_j$ linear
  coordinates on the fiber) we have
  \begin{equation}\label{lift-explizit}
  v=\sum_i f_i(x)\frac{\partial}{\partial x_i}
  \ \mapsto\
  \tilde v=\left(\sum_i f_i(x)\frac{\partial}{\partial x_i}\right)
  +\sum_{i,j,k}A_{i,j}^kf_k(x)y_i\frac{\partial}{\partial y_j}
  \end{equation}

  For details on the theory of connections,
  see e.g. \cite{Atiyah},\cite{Wells}.

  \begin{fact}
    Let $X$ be a manifold, $E$ a vector bundle over $X$ with a
    connection $\nabla$.

    Then every vector field on the manifold $E$ decomposes as
    the sum of a horizontal (with respect to $\nabla$) and a
    vertical vector field.
  \end{fact}
  \begin{proof}
    This follows from the direct sum decomposition $T_E=H\oplus V$.
  \end{proof}

\begin{proposition}
Every holomorphic vector bundle $E$ over a Stein manifold $X$ admits a holomorphic connection.
\end{proposition}

\begin{proof}
The obstruction to the existence of a holomorphic connection on $E$ is given by a cohomology class in $H^1(X, \Omega^1 \otimes End(E))$. The sheaf $\Omega^1 \otimes End(E)$ is coherent and $X$ is Stein, hence Cartan's Theorem B guarantees that the cohomology group is trivial and there is no obstruction to the existence of a holomorphic connection.
\end{proof}

Now that we can lift vector fields, we need to check that completeness is preserved.

\begin{proposition}\label{horiz}
Let $E \to X$ be a holomorphic vector bundle equipped with a holomorphic connection.

Then for every complete vector field on $X$ the lift to $E$ which is horizontal with respect to the connection is complete, too.
\end{proposition}

\begin{proof}
  Let $V=\sum_i f_i(x)\dd {x_i}$ be (the local expression for) a complete vector field on the base with
  associated flow $\Phi(x,t)$, i.e., $ \Phi'(x,t)=f(\Phi(x,t))$.

  Locally, on some neighborhood $U\subset X$ we may trivialize $E$ such that
  $\nabla(\sigma)=d\sigma+\alpha \sigma$
  where $\alpha$ is a holomorphic function on $U$ with
  values in $Mat(n\times n,\CC)$. The horizontal lift of $V$ is given by Equation \eqref{lift-explizit} above.

  Hence, we are looking for a function $H:U \times \CC^r \times \CC \to \CC^r$ such that
  \[
  H_i'(x,s,t)=\sum_{i,j,k}\alpha^i_{jk}(\Phi(x,t))v_j(\Phi(x,t))H_k(x,s,t)
  \]
  But this is a {\em linear} differential equation. Thus there is a global
  solution (global on the fiber).
  
\end{proof}
  We conclude this section with a lemma which will be used later.
Loosely speaking, a $\pi$-compatible vector field is one that maps $\pi$-fibers to $\pi$-fibers; see Section \ref{incompatible} for more details.

\begin{lemma} \label{horizontal}
Let $E \to X$ be a holomorphic vector bundle equipped with a holomorphic connection over a Stein manifold with the density property $X$.

Then every horizontal vector field on $E$
which is $\pi$-compatible
 can be approximated by linear combinations of complete vector fields. 
\end{lemma}

\begin{proof}
  A $\pi$-compatible horizontal vector field on $E$ is the lift of a vector field $V$ on $X$. Because of the density property and Fact \ref{spanLie}, $V$ can be approximated by a linear combination of complete vector fields. Since lifting a vector field is a linear operation, the same linear combination of the lifts of these vector fields will approximate our original vector field on $E$.
\end{proof}

\section{The Euler vector field}\label{sect-euler}

Let $E\to X$ be a complex vector bundle. Then there is a natural
$\C^*$-action on $E$ with a corresponding vector field,
the {``\em Euler vector field''} $\Theta$.

In local coordinates: Let $(x_i)_i$ be coordinates on the
base space $X$ and let $(y_j)_j$ be fiberwise linear coordinates
on $E$.
Then
\[
\Theta=\sum_j y_j\frac{\partial}{\partial y_j}
\]

The vector bundle structure is completely determined by the Euler vector field:
The zero-section coincides with the vanishing locus of $\Theta$,
$\pi(p)=\pi(q)$ for $p,q\in E$ iff the closures of the $\C^ *$-orbits
through $p$ and $q$ intersect.
A function $\lambda$ on $E$ is fiber wise linear if and only
if $\Theta\lambda=\lambda$.

\begin{lemma}\label{euler-hori-commute}
  Let $E\to X$ be a holomorphic vector bundle with a holomorphic connection
  $\nabla$ and let $\Theta$ be the Euler vector field.

  Then $[\Theta,V]=0$ for every horizontal vector field $V$.
\end{lemma}
\begin{proof}
  This follows from an easy calculation using \eqref{lift-explizit}.
\end{proof}

\section{Vector fields (in)compatible with the vector bundle structure}%
\label{incompatible}

\begin{definition}\label{def-compat}
  Let $\pi:E\to X$ a holomorphic vector bundle on a complex manifold $X$.

  \begin{itemize}
  \item
  An automorphism $\phi$ of the complex manifold $E$ is called
  {\em ``compatible with the bundle structure''},
  or simply ``$\pi$-compatible''
  if $\phi$ maps $\pi$-fibers to $\pi$-fibers
  (In other words: if
  $\forall p,q\in E:\pi(p)=\pi(q)\implies \pi(\phi(p))=\pi(\phi(q))$.)
\item
A vector field $v$ on $E$ (not necessarily complete)
is called ``$\pi$-compatible'' if for every
$p,q\in E$ with $\pi(p)=\pi(q)$ we have $\pi_*(v_p)=\pi_*(v_q)$.
  \end{itemize}
\end{definition}

In the above setup an automorphism or a vector field
is called {\em incompatible} if it is not compatible.

  If $X$ is smooth and therefore normal, being ``compatible''
is equivalent to the existence of an automorphism $\psi$ of $X$ such that
the diagram
\[
\begin{tikzcd}
E \arrow{r}{\phi} \arrow{d}{\pi} & E \arrow{d}{\pi}\\
X \arrow{r}{\psi} & X \\
\end{tikzcd}
\]
is commutative.

Note that we do not require the map $\phi$ to be fiberwise linear.
Hence the notion ``{\em $\pi$-compatible automorphism}''
is weaker than the notion
of a vector bundle automorphism.

For a vector field $v$ on $E$
``$\pi$-compatibility'' 
is equivalent to the existence of a vector field
$w$ on $X$ with $\pi_*(v_p)=w_{\pi(p)}\ \forall p\in E$.

If a $\pi$-compatible vector field is also complete, the associated
flow consists of $\pi$-compatible automorphisms.

\begin{lemma}\label{lem5.1}
  Let $v$ be a $\pi$-compatible vector field
  on $E$ and let $f$ be a holomorphic function
  on $E$.
  Assume that $v$ is not everywhere tangent to the $\pi$-fibers (i.e., $\pi_*v\not\equiv 0$).
  
  Then $fv$ is $\pi$-compatible if and only if $f$ is constant
  along the $\pi$-fibers.
\end{lemma}

\begin{proof}
  Assume that $f$ is not constant along the $\pi$-fibers.
  
  The set $\Omega_1$ of $p\in E$ where $f$ is not constant along
  the $\pi$-fiber is open and dense
  in $E$.
  The set $\Omega_2$ of points where $v$ is not tangent to the $\pi$-fiber is
  also open and dense.

  Hence the intersection of these two sets $\Omega_1\cap\Omega_2$
  is still dense and in
  particular not empty. Let $p$ be a point in $\Omega_1\cap\Omega_2$.
  Choose $q\in\Omega_1\cap\Omega_2$
  with $\pi(p)=\pi(q)$, but $f(p)\ne f(q)$.
  (This is possible, since $p\in\Omega_1$.)
  By construction, we have $\pi_*(v_p)\ne 0$ and
  $\pi_*(v_p)=\pi_*(v_pq)$.
  
  We obtain
  \[
  \pi_*(fv_p)=f(p)\underbrace{\pi_*(v_p)}_{\ne 0}
  \ne f(q)\pi_*(v_p)=f(q)\pi_*(v_p)=\pi_*(fv_q).
  \]
  Thus $fv$ can not be $\pi$-compatible unless $f$ is constant along the
  $\pi$-fibers.

  The opposite direction is obvious.
\end{proof}

\begin{lemma}
Let $V$ be a complete vector field on $E$. Then the associated
flow
\begin{itemize}
\item
  consists of vector bundle automorphisms if $[\Theta,V]=0$.
\item
  consists of automorphisms $\phi$ which stabilize each fiber
  (i.e.~$\pi\circ\phi=\pi$) and act fiberwise
   as translations if $[\Theta,V]=-V$.
\end{itemize}
\end{lemma}

\begin{proof}
  This follows easily from explicit calculations in local
  coordinates.
\end{proof}

\begin{proposition}\label{from-auto-to-vf}
  Let $E\to X$ be a vector bundle over a Stein manifold $X$.
  
  If there exists a $\pi$-incompatible automorphism $\phi$, there also
  exists a $\pi$-incompatible complete vector field $v$.
\end{proposition}

\begin{proof}
  Let $\Theta$ denote the Euler vector field (see \S\ref{sect-euler}).
  If $\phi_*\Theta$ is $\pi$-incompatible, we are done.
  Thus we may assume that $\phi_*\Theta$ is $\pi$-compatible.
  Let $F$ be a $\pi$-fiber such that $\phi^{-1}(F)$ is not a fiber.
  Then we choose a holomorphic function $f$ on $X$ such that $f\circ\pi$
  is not constant along $\phi^{-1}(F)$. This implies that
  $f\circ\pi\circ\phi$ is not constant on $F$. Now
  Lemma \ref{lem5.1}
  implies that $\phi_*\left( (f\circ\pi)\Theta\right)$ is a
  complete vector field on $E$ which is incompatible with $\pi$.
\end{proof}

\begin{corollary}
  Let $\pi:E\to X$ be an algebraic vector bundle and assume
  that $E$ is a flexible smooth affine variety.

  Then there
  exists a $\pi$-incompatible complete holomorphic vector field $v$
  on $E$.
\end{corollary}

\begin{proof}
  The automorphism group of every flexible smooth affine variety
  is $n$-transitive for every
  $n\in\N$. Choose $p,q_1,q_2\in E$ with
  $\pi(p)=\pi(q_1)\ne \pi(q_2)$ and $p\ne q_1$.
  Since $\Aut(E)$ (the group of variety automorphisms of $E$)
  is $2$-transitive, there exists an automorphism $\phi$ of the  variety $E$
  such that $\phi(p)=p$, $\phi(q_1)=\phi(q_2)$.
  Now $\pi(\phi(p))\ne\pi(\phi(q_1))$ although $\pi(p)=\pi(q_1)$.
  Thus $\phi$ is $\pi$-incompatible and the statement
  follows from Proposition~\ref{from-auto-to-vf}.
\end{proof}

\section{Preparations}

\subsection{Existence of certain vector fields}

  \begin{lemma}\label{ex-stein-vf}
    Let $X$ be a Stein manifold, $p\in X$, $u\in T_pX$ and
    let $W$ be a  holomorphic vector field with $W(p)\ne 0$.

    Then there exists a holomorphic vector field $V$ on $X$
    with $[V,W](p)=u$.
  \end{lemma}

  \begin{proof}
    In suitably chosen local coordinates in some neighborhood of $p$
    we have
    \[
    W=\dd{x_1},\quad u=(u_1,\ldots,u_n)=\sum_j u_j\dd{x_j}\quad (u_j\in\C)
    \]
    We choose $V$ such that
    \[
    V=-\sum_j u_jx_1\dd{x_j}+\text{ terms vanishing of order $\ge 2$ in $p$}
    \]
    (This is possible, because $X$ is Stein.)
    Then
    \[
    [V,W]=\sum_j u_j\dd{x_j}+\text{ terms vanishing in $p$}
    \]
  \end{proof}
  
\subsection{Runge Theory}

\begin{proposition}
  Let $\pi:E\to X$ be a holomorphic vector bundle over a Stein manifold $X$
  and let $\Omega\subset X$ be a Runge domain
  (i.e.~every holomorphic function on $\Omega$ may be
  approximated by holomorphic functions on $X$).

  Then every section $s\in\Gamma(\Omega,E)$ may be approximated
  by sections on $X$. (We assume that $E$ is equipped with some
  fixed hermitian metric.)
\end{proposition}

\begin{proof}
This is a consequence of \cite[Theorem 5.4.4]{ForstnericBook}.
\end{proof}




\begin{theorem}\label{dense-sections} 
  Let $\pi:E\to X$ be a holomorphic vector bundle on a Stein manifold $X$.
  Let
  $(\sigma_j)_{j\in I}$
  be a 
  family of  holomorphic sections in $E$ which span
  $E$ everywhere.

  Then the $\Olo_X$-module generated by the $\sigma_i$ is
  dense in the space of global sections of $E$.
\end{theorem}

\begin{remark}
  Density is the best one may hope for, equality is not always satisfied.
  For example, let $X=\C$, $E$ the trivial line bundle and consider
  \[
  \sigma_k=\frac{e^{2\pi i z}-1}{z-k}
  \]
  for $k\in\Z$.
  Then every holomorphic function ($=$ section in the trivial line bundle)
  which is contained in the $\Olo_X$-module generated by the
  $(\sigma_k)_{k\in\Z}$
  must vanish on almost all integers, hence this $\Olo_X$-module does not
  equal $\Olo(\C)$, although it is dense.
\end{remark}

\begin{proof}
  Fix a Runge exhaustion $U_k$, i.e., we choose
  holomorphically convex open subsets $U_k$
  of $X$ such that $U_k$ is relatively compact in $U_{k+1}$,
  $U_k$ is Runge in $X$ and $X$ equals the union of all $U_k$.
  
  (For example, if $\rho$ is a strictly plurisubharmonic exhaustion
    function, we may choose $U_k=\{x\in X:\rho(x)<k\}$.)

  Fix $p\in X$ and $k\in\N$. Choose $\sigma_1,\ldots,\sigma_r$
  within the given family such that
  $\sigma_1(p),\ldots,\sigma_r(p)$ generate $E_p$.
  Then these sections generate $E$ outside an analytic subset $S_1$
  of codimension at least $1$.
   
   Now we may proceed by recursion: Choosing one point on each irreducible
   component of $S_1\cap U_k$ and then enough sections such that they generate
   at these points. Since $U_k$ is relatively compact,
   there are only finitely many such irreducible components.
   Hence finitely many additional sections are enough.
   In this way we may increase
   the codimension of the singular set
   step by step, finally verifying: {\em For every $k\in\N$ there exists
     a finite subfamily $\{ \sigma_j: j\in I_k\}$ such that these
     sections generate $E_p$ for every $p\in U_k$.}

   Let $N_k=\#I_k$. Let $I_k=\{\sigma_1,\ldots,\sigma_{N_k}\}$.
   We consider the sheaf morphism $\Olo^{N_k}\to\sE$ given by
   \[
   (f_1,\ldots,f_N)\mapsto \sum_{j=1}^{N_k} f_j\sigma_j
   \]

   Let $\Ks$ denote the kernel sheaf.
   For every $k$ we obtain a long exact
   sequence
   \[
   \to \Gamma(U_k,\Olo^{N_k})\longrightarrow\Gamma(U_k,E)
   \longrightarrow H^1(U_k,\Ks)\to\ldots
   \]
   Since $U_k$ is Stein
   and $\Ks$ is coherent,
   the cohomology group $H^1(U_k,\Ks) $
   must be trivial. Thus
   $\Gamma(U_k,\Olo^N)\to\Gamma(U_k,E)$ is surjective.
   Hence every section of $E|_{U_k}$ is contained in the
   $\Olo_{U_k}$-module generated by the $\sigma_i$.

   Let $s\in\Gamma(X,E)$. For every $k\in\N$ the section $s$
   is contained in the $\Olo_{U_k}$-module generated by the $\sigma_j$.

   Fix a hermitian metric on $E$. 
   Since $U_k$ is Runge in $X$, we may find $s_k\in\Gamma(X,E)$
   such that
   \[
   ||s-s_k||_{\overline{U_{k-1}}}\le \frac 1k
   \]
   Then $s=\lim_{k\to\infty}s_k$ for the topology of locally uniform
   convergence.
   Therefore 
   $\Gamma(X,E)$ equals the
   closure of the $\Olo_X$-module generated by the $\sigma_j$.
\end{proof}

\subsection{Fréchet Spaces}

\begin{theorem}[Peter-Weyl-Theorem]\label{pweyl}
  Let $V$ be a Fr\'echet space and let $G$ be a reductive group
  acting linearly and continuously on $V$.
  An element $v\in V$ is called $G$-finite if the $G$-orbit through   $v$
  spans a finite-dimensional vector subspace of $V$.

  Then the set of all $G$-finite elements is dense in $V$.
\end{theorem}

\begin{proof}
See \cite{PeterWeyl} for the original reference pertaining only to compact groups or the textbook \cite{AlgGeoIV} for the statements about reductive groups.
\end{proof}

\begin{corollary}
  Let $\pi:E\to X$ be a vector bundle, $\Theta$ the Euler vector field
  and let $H$ be a closed subspace of $\Gamma(E,TE)$ which is invariant
  under
  \[
  ad(\Theta): v \mapsto [\Theta,v].
  \]

  For $k\in\Z$, let
  \[
  H_k=\{v\in H:[\Theta,v]=kv\}.
  \]

  Then
  \[
  \oplus _{k\ge -1} H_{k}
  \]
  is dense in $H$.
\end{corollary}

  \begin{proof}
    This follows from the Peter-Weyl-Theorem,
    because $H_k$ equals the set of $v\in H$ for which
    the natural $\C^ *$-action operates as $\lambda:v\to\lambda^ kv$.
  \end{proof}

  \section{Vertical vector fields}

As mentioned in the introduction, we will prove that vector fields preserving the projection can be approximated by Lie combinations of complete vector fields. We already proved that horizontal vector fields preserving the projection can be approximated by Lie combinations of complete vector field. We will now show the same result for vertical vector fields. These aspects will be put together in Proposition \ref{approximable}.

\begin{remark}\label{Lie-vs-lin}
  If a vector field $v$ on a  manifold can be approximated
  by Lie combinations of complete vector fields, it can already
  be approximated by linear combinations. This holds, because
  \[
    [v,w]=\lim_{t\to 0}\frac{\phi_t^* w-w}{t}
  \]
  if $v$ is complete with flow $\phi_t$.
\end{remark}

  \begin{definition}\label{def-adm}
  Let $\pi:E \to X$ be a holomorphic vector bundle
  of rank $r$ over a Stein manifold $X$.

  An {\em ``admissible family''} (or ``admissible frame'')
  is a family $(\sigma_1,\ldots,\sigma_r;f)$

  such that
  \begin{itemize}
  \item
    Every $\sigma_j$ is a section in $E$.
  \item
    $\sigma_1\wedge\ldots\wedge\sigma_r$ is not identically zero.
  \item
    $f$ is a non-constant holomorphic function on $X$.
  \item
    $f\left(\sigma_1\wedge\ldots\wedge\sigma_r\right)^ {-1}$ is a
    holomorphic section of $(\det E)^ *$.
  \end{itemize}
  \end{definition}

  We observe that by the last condition $\sigma_1(x),\ldots,\sigma_r(x)$
  is a vector space basis of the fiber $E_x=\pi^{-1}(x)\simeq\C^ r$
  for every $x\in X$ with
  $f(x)\ne 0$.
  
  \begin{lemma}\label{V-1}
  Let $\pi:E \to X$ be a holomorphic vector bundle
  of rank $r$ over a Stein manifold $X$.

  Let $S$ be a countable subset of $X$. Then there exists
  an admissible family (in the sense of Definition~\ref{def-adm})
  with
  \[
  f(x)\ne 0 \ \forall x\in S
  \]
\end{lemma}

\begin{proof}
  The space of global holomorphic sections $V=\Gamma(X,E^ r)$
  is a Fr\'echet vector space with respect to the topology of locally uniform
  convergence. Hence its topology is completely metrizable and $V$
  satisfies the Baire property. Thus the intersection of countably many open
  dense subsets is still dense and in particular not empty.

  For each $x\in S$ we consider the set $W_x$ of those
  $(\sigma_1,\ldots,\sigma_r)\in V$
  for which $\sigma_1(x),\ldots,\sigma_r(x)$
  forms a basis of the vector space $E_x$.

  Then each $W_x$ is open and dense in $V$, implying:
  \[
  \cap_{x\in S}W_x\ne\{\}
  \]
  We choose an element
  \[
  (\sigma_1,\ldots,\sigma_r)\in\cap_{x\in S}W_x
  \]
  Consider the set $C$ of all $x$ in which these sections fail to
  generate, i.e., the zero-locus of
  $\sigma_1\wedge\ldots\wedge\sigma_r$.
  This is a closed analytic set which does not intersect $S$.
  Taking multiplicities into account, we may regard $C$ as a divisor.
  Let $F$ denote the space of those holomorphic
  functions on $X$ which vanish on each irreducible component
  of $C$ with at least the same multiplicity
  as $\sigma_1\wedge\ldots\wedge\sigma_r$, i.e., $div(F)\ge C$ as divisors.
  
  For each $x\in S$, the set $\{f\in F:f(x)\ne 0\}$ is open and dense.
  Thus there is an element $f\in F$ with
  $\forall x\in S:f(x)\ne 0$.

  By construction, $(\sigma_1,\ldots,\sigma_r;f)$ is an admissible
  family.
\end{proof}

\begin{lemma}\label{fin-adm}
  There are finitely many admissible families
  \[
  A_k=(\sigma_1^{k},\ldots,\sigma^{(k)}_r;f^ {(k)})
  \]
  such that
  \[
  \forall x\in X:\exists k: f^ {(k)}(x)\ne 0
  \]
\end{lemma}

\begin{proof}
  We start with one admissible family $A_1$.

  Given finitely many admissible families $A_1,\ldots,A_k$, let
  $Z_k$ denote the common zero locus of the $f^ {(k)}$.
  We choose a countable set $\Sigma_k$ which intersects every irreducible
  component of $Z_k$ and apply Lemma~\ref{V-1} with $S=\Sigma_k$.
  This yields another family $A_{k+1}$ such that $\dim Z_{k+1}<\dim Z_k$.

  Thus we may proceed recursively until we found enough families to arrive
  at an empty common zero locus.
\end{proof}

\begin{lemma}\label{vf-xi}
  Let $\pi:E\to X$ be a vector bundle of rank $r$ and let
  $(\sigma_r,\ldots,\sigma_r;f)$ be an admissible family.

  Then there exists a continuous linear operator $\Xi$ from the
  space of vector fields on $\C^ r$ to the space of vertical vector fields
  on $E$ such that:

  \begin{enumerate}
  \item
    $\Xi$ preserves the natural grading (induced by the natural
    $\C^ *$-action).
  \item
    For every $x\in X$ with $f(x)\ne 0$ the evaluated
    map $v\mapsto \Xi(v)(x)$ equals the map induced by a linear
    isomorphism of $\C^ r$ and $E_x\simeq\C^ r$.
  \item
    $\Xi(v)$ is complete if $v$ is complete.
  \end{enumerate}
\end{lemma}

\begin{proof}
  We define $X^ *=\{x\in X:f(x)\ne 0\}$ and
  $\Phi: X^*\times\C^ r\to E$ as
  \[
  \Phi(x;t)=\frac{1}{f(x)}\sum_{j=1}^ r t_j\sigma_j(x)
  \]
  and $\Phi_x(t)=\Phi(x,t)$.

  For a vector field
  \[
  v=\sum_{k=1}^r h_k(t)\dd{t_k}
  \]
  we obtain a vertical vector field on $E|_{X^ *}$ as
  \begin{align*}
    (D\Phi_x)(v)&= \sum_{k=1}^r h_k(t)(D\Phi_x)\left(\dd{t_k}\right)\\
    & =\sum_{k,l=1}^r h_k(t)(D\Phi_x)_{k,l}\cdot\left(\dd{s_l}\right)
      \text{ with $s=(D\Phi_x)(t)$}
    \\
    & =\sum_{k,l=1}^r h_k(\Phi_x^{-1}(s))(D\Phi_x)_{k,l}\cdot \left(\dd{s_l}\right)\\
  \end{align*}

  We need some linear algebra: For a $r\times r$-matrix $S$ Cramer's rule states
  $S^{-1}=\frac{1}{\det(S)}S^ \#$. If $f$ is a scalar, we deduce
  \[
  (S\frac 1f)^ {-1}=\frac{1}{\det(S\frac 1f)}(S\frac 1f)^ \#
  =\frac{1}{f^{-r}\det(S)}f^{1-r}(S)^ \#=\frac{f}{\det(S)}S^ \#
  \]

  In our setting we have $\Phi_x=\frac 1{f(x)}S(x)$
  where $S(x)$ is the $r\times r$-matrix
  formed by the $\sigma_j(x)$ and we may deduce that
  $x\mapsto \Phi_x^{-1}$ which a priori is defined only on $X^*$
  actually extends holomorphically to all of $X$.

  On the other hand, by construction $x\mapsto f(x)\Phi_x$ is holomorphic
  on the whole $X$.

  Therefore
  \begin{equation}\label{def-vvf}
  \Xi(v)=\tilde v=\left(f(x)\right)^ m(D\Phi_x)(v)
  \end{equation}
  defines a global vertical holomorphic vector field on $E$ for $m\ge 1$.

  Immediate from the construction we have properties $(1)$ and $(2)$.
  Property $(3)$ holds on $X^ *$, because $\Phi_x$ is an isomorphism
  where $f(x)\ne 0$.

  Taking $m\ge 2$ in equation (\ref{def-vvf})
  the vector field $\tilde v$ will vanish
  on all fibers $E_x$ with $f(x)=0$ and therefore trivially be complete.
\end{proof}

\begin{proposition} \label{vertical}
Let $\pi:E \to X$ be a holomorphic vector bundle over a Stein manifold $X$.

If $r=rank(E)\ge 2$,
then every vertical vector field on $E$ can be
approximated by Lie combinations of complete vertical vector fields.
\end{proposition}

\begin{proof}
  Let $v$ be a vertical vector field on $E$. Due to the Peter-Weyl-theorem
  (Theorem~\ref{pweyl}) we may approximate $v$ by finite sums
  of homogeneous vector fields. Hence it suffices to consider the case
  where $v$ is homogeneous, i.e., $[E,v]=kv$ for some degree $k\in\Z$.

  Vertical vector fields of a fixed degree $k$ can be regarded as
  sections in a coherent $\Olo_X$-module sheaf on $X$.
  
  Lemma~\ref{fin-adm} gives us finitely many admissible
  families
  \[
  A_k=(\sigma_1^{n},\ldots,\sigma^{(n)}_r;f^ {(n)}):\quad \forall x\in X:\exists k:
  f^ {(n)}(x)\ne 0
  \]

  With Lemma~\ref{vf-xi} it follows that for every fixed $k$
  there are finitely many vertical
  vector fields $w_\alpha$ on $E$ such that
  \begin{enumerate}
  \item
    Every $w_\alpha$ is induced by a map $\Xi$ associated to
    an admissible family as
    described in Lemma~\ref{vf-xi}.
  \item
    For every $x$ the vector space of vector fields
    of degree $k$ on $E_x$ is spanned by the $w_\alpha$.
  \end{enumerate}

  By Andersen-Lempert theory every vector field on $\C^ r$ can be
  approximated by linear combinations of complete vector fields
  (taking into account \ref{Lie-vs-lin}).

  Thus for every $w_\alpha$ we can find a vector field $u$ on $\C^ r$ and a
  sequence of sums of complete vector fields $u_n$  on $\C^ r$
  with
  \[
  w_\alpha=\Xi(u)\text{ and }\lim_{n\to\infty}u_n=u
  \]
  Therefore $w_\alpha=\lim_{n\to\infty}\Xi(u_n)$. Note that the vertical vector
  fields $\Xi(u_n)$ are complete by Lemma~\ref{vf-xi}, (3).

  Thus every $w_\alpha$ is approximable. On the other hand, since $X$ is Stein,
  the vertical vector fields of fixed degree $k$ form a coherent sheaf and
  this sheaf is everywhere generated by the $w_\alpha$,
  we know that every vertical vector field of degree $k$ may be
  represented as a finite $\Olo_X$-linear combination of the $w_\alpha$.

  Since we discuss {\em vertical} vector fields on $E$, finite
  $\Olo_X$-linear combinations of complete resp.~approximable
  vertical vector fields are still complete resp.~approximable.

  This completes the proof.
  \end{proof}

Now that we established these facts, let us proceed with one of the two main ingredients of the proof of Theorem \ref{mainrank2}.

\begin{proposition} \label{approximable}
Let $E \xrightarrow[]{\pi} X$ be a holomorphic vector bundle of rank $\geq 2$ over a Stein manifold $X$ with the density property.

Then every $\pi$-compatible vector field on $E$
can be approximated by a linear combination of complete vector fields.
\end{proposition}
\

\begin{proof}
A $\pi$-compatible vector field $W$ on $E$  can be decomposed as the sum of a vertical vector field $V$ and a $\pi$-compatible horizontal vector field $H$ .

The conclusion is obtained by applying Proposition \ref{vertical} and Lemma \ref{horizontal} to $V$ and $H$, respectively.
\end{proof}

\begin{theorem}\label{no-inter}
  Let $\pi:E\to X$ be a holomorphic vector bundle
  over a Stein manifold $X$.

  Let $L$ be the Lie algebra of all holomorphic vector fields on the manifold
  $E$ and let $P$ denote the subalgebra of $\pi$-preserving vector fields.

  Then there does not exist any intermediate closed Lie algebra, i.e., if
  $H$ is a closed Lie subalgebra of $L$ with $P\subset H$, then either
  $H=P$ or $H=L$.
\end{theorem}

\begin{proof}
  Fix $p\in X$. In local coordinates near $p$, vector fields on $E$ may be
  written as
  \begin{equation}\label{loc-vf}
  v=\sum_{\alpha,k} f_{\alpha,k}(x)y^\alpha\frac{\partial}{\partial x_k}
  + \sum_j g_{j}(x,y)\frac{\partial}{\partial y_j}
  \end{equation}
  where $x=(x_1,\ldots,x_n)$ are local coordinates on the base
  and $y=(y_1,\ldots,y_r)$ local coordinates in the
  fiber direction with $\alpha=(\alpha_1,\ldots,a_k)$ being a
  multi-index. As usual, $|\alpha|=\sum_j\alpha_j$.

  Note that $v\in P$ if and only if
  $f_{\alpha,k}\equiv 0$ for $|\alpha|>0$.
    
  Thus any element $w$  in $P$ may locally be written in the form
    \begin{equation}\label{loc-pvf}
  w=\sum_{k} f_{0,k}(x)\frac{\partial}{\partial x_k}
  + \sum_j g_{j}(x,y)\frac{\partial}{\partial y_j}
    \end{equation}
  
  Consequently,
  elements in $L/P$ may be represented as
  \begin{equation}\label{loc-qvf}
  \sum_{\substack{\alpha,k\\ \alpha\ne 0}}
  f_{\alpha,k}(x)y^\alpha\frac{\partial}{\partial x_k}
  \end{equation}

  Assume that $H\ne P$.
  We consider
  the adjoint action of the Euler vector field
  \[
  \Theta=\sum_j y_j\frac{\partial}{\partial y_j}
  \]
  For $v$ as in \eqref{loc-vf} we have
  \[
  [\Theta,v]=
  \sum_{\alpha,k} |\alpha| f_{\alpha,k}(x)y^\alpha\frac{\partial}{\partial x_k}
  + \sum_j (\theta(g_j)-g_j)\frac{\partial}{\partial y_j}
  \]

  Hence $[\Theta,v]\in P$ for $v\in P$ and we obtain an induced
  action $ad(\Theta)$ on $L/P$ via $v\mapsto [\Theta,v]$.
  This action is given as
  \[
  ad(\Theta)(w)=
  \sum_{\substack{\alpha,k\\ \alpha\ne 0}}
  |\alpha| f_{\alpha,k}(x)y^\alpha\frac{\partial}{\partial x_k}
  \text{ for }
  w=  \sum_{\substack{\alpha,k\\ \alpha\ne 0}}
  f_{\alpha,k}(x)y^\alpha\frac{\partial}{\partial x_k}
  \]
  Note that $\Theta$ is complete and yields a $\C^*$-action on $L/P$
  given as
  \[
\C^ *\ni   \lambda:w\mapsto
  \sum_{\substack{\alpha,k\\ \alpha\ne 0}}
  \lambda^{|\alpha|} f_{\alpha,k}(x)y^\alpha\frac{\partial}{\partial x_k}
  \]
  with $w$ as in \eqref{loc-qvf}.
  $\C^*$ is a reductive group and $L$ is a Fr\'echet space.
  Hence we may apply the Peter-Weyl Theorem (Theorem \ref{pweyl}).
  This theorem implies that, given a $\C^ *$-action $\mu$ on a Fr\'echet space
  $V$, the subspace $\oplus_{d\in\Z} V_d$ is dense in $V$ where
  \[
  V_d=\{v\in V: \mu(\lambda)(v)=\lambda^d v\}.
  \]
  Since $L$ is a Fr\'echet space and $P$ and $H$ are closed invariant subspaces,
  the quotient space $H/P$ is a Fr\'echet space  on which $\C^*$ acts.

  It follows that if we take $w\in H/P$ and write it as in \eqref{loc-qvf}, we have
  \[
  \sum_{|\alpha|=c}f_{\alpha,k}(x)y^\alpha\frac{\partial}{\partial x_k}\in H/P
  \ \forall c\in\N.
  \]

  In other words, if $H\ne P$, then
  $H/P$ contains a non-zero vector field $w$ which is homogeneous in the
  $y$-graduation, i.e., $\exists m\in\N\setminus\{0\}$:
  \[
  H\ni w=\sum_{\substack{\alpha,k\\ |\alpha|=m}}
  f_{\alpha,k}(x)y^\alpha\frac{\partial}{\partial x_k}
  \quad \text{ modulo $P$}
  \]

  The $\C^*$-action induced by the Euler vector field yields
  gradings on $L$, $H$, $P$ and the algebra of vertical vector fields
  $V$.

  For any $d\in \N$ the vector spaces $L_d$, $P_d$ and $V_d$ can be realized
  as spaces of global section on $X$ of coherent sheaves $\sL_d$,
  $\sP_d$, $\sV_d$, e.g., for any open subset $U\subset X$, we define
  $\sP_d(U)$ as the set of all vector fields $w$ on $\pi^ {-1}(U)\subset E$
  which are in the form \eqref{loc-pvf} and satisfy $[\Theta,w]=dw$.
  
  Similarily the quotient space $(L/P)_d=L_d/P_d$ arises as space of
  global sections
  for a coherent sheaf $\Sc_d$ ($d\in \N$).

  Using this notions, $H\ne P$ implies that
  there exists a number $d\in\N$
  such that $(L/P)_d$ has non-zero intersection with $H/P$.

  Let us now consider vector fields which locally may be written as
  \[
  u=\sum_j c_j(x)\frac{\partial}{\partial y_j}
  \]
  The family of these vector fields define a coherent sheaf
  $\mathcal A=\sV_1$ on $X$. Since $X$ is Stein, this sheaf is spanned
  by global sections. Fix a non-zero such vector field $u$.

  The adjoint action $ad(u):v\mapsto [u,v]$ preserves $P$ and
  defines a morphism of the graded Lie algebra $L/P$ of degree $-1$,
  i.e.,
  \[
  ad(u):(L/P)_{k+1}\mapsto (L/P)_k
  \]

  Observe that $u\in P$, because it is a vertical vector field.
  Therefore
  \[
  ad(u):(H/P)\mapsto (H/P)
  \]

  Let
  \[
  v=
  \sum_{\substack{\alpha,k\\ |\alpha|=d}}
  f_{\alpha,k}(x)y^\alpha\frac{\partial}{\partial x_k}
  \]
  represent a non-zero homogeneous element of $H/P$.

  Since $\mathcal A$ is spanned by global sections, there is a vector
  field $u\in P$ as above with
  \[
  \sum_{\substack{\alpha,k\\ |\alpha|=d}}
  f_{\alpha,k}(x)u\left(y^\alpha\right)\frac{\partial}{\partial x_k}\not\equiv 0
  \]

  It follows: {\em If $H/P$ contains a non-zero  homogeneous element of
  degree $d+1$, it also contains a non-zero  homogeneous element of
  degree $d$.}

  By repeated application of this argument we deduce that
  $H/P$ contains a non-zero homogeneous element of degree $1$.
  
  In other words,  for the grading of $L/P=\oplus_d (L/P)_d$
  defined by the eigenspace
  decomposition with respect to $ad(\Theta)$, we have
  $H/P$ has non-zero intersection with $(L/P)_1$.

  Next we claim: {\em There is such a homogeneous element $v$ of degree $1$
  in $H/P$ which may locally be written as
  \[
  v=
  \sum_{\substack{\alpha,k\\ |\alpha|=1}}
  f_{\alpha,k}(x)y^\alpha\frac{\partial}{\partial x_k}
  \]
  such that there exists an index $\alpha$ (with $|\alpha|=1$)
  such that $f_{\alpha,k}(p)\ne 0$.}

  For any $v$ as in the claim we call the minimum of the vanishing orders
  in $p$
  of the components $f_{\alpha,k}$ ($|\alpha|=1$) simply the {\em order} of $v$.
  Thus our goal is to show that $H$ contains an element $v$ of degree $1$ and
  order $0$.
  Let $N$ denote the minimal number $N\in\N_0$ for which $H$ contains an
  element of degree $1$ and order $N$ and fix such an element $v$.

  Next let $w$ be a vector field on $X$ which does not vanish in $p$.
  In appropriately chosen local coordinates near $p$ on $X$ we have
  \[
  w=\frac{\partial}{\partial x_1}
  \]
  and
  \[
  v=
  \sum_{\substack{\alpha,k\\ |\alpha|=1}}
  f_{\alpha,k}(x)y^\alpha\frac{\partial}{\partial x_k}
  =
  \sum_{j,k} f_{j,k}(x)y_j\frac{\partial}{\partial x_k}
  \]
  with $N$ being the minimal vanishing order of the $f_{\alpha,k }$.
  
  Next let $\tilde w$ be the lift of $w$ with respect to some holomorphic
  connection on the vector bundle $E\to X$. Then
  \[
  \tilde w=\frac{\partial}{\partial x_1}
  +\sum_{j,l}A_{j,l}(x)y_j\frac{\partial}{\partial y_l}
  \]
  where the $A_{j,l}$ are holomorphic functions which depend on the
  choice of the holomorphic connection.
  
  As a horizontal vector field, $\tilde w\in P$.
  Hence $[\tilde w,v]\in H$. Now
  \[
    [\tilde w,v]=\sum_{j,k} \frac{\partial f_{j,k}}{\partial {x_1}}y_j
    \frac{\partial}{\partial x_k}\text{ $+$ terms vanishing of order $\ge N$}
  \]
  By the minimality of $N$, it follows that all the $f_{j,k}$
  vanish of order at least
  $N+1$ along the integral curve of $w$ (i.e.~in the $x_1$-direction).

  But $w$ was arbitrary vector field on $X$ (with $w(p)\ne 0$).
  Hence we obtain: {\em If $N>0$, then for every integral curve of every
    vector field not vanishing in $p$ the restriction of every $f_{j,k}$
    to this curve vanishes of order at least $N+1$ in $p$.}
  The vanishing order of a holomorphic function equals the minimum of
  the vanishing order of the restriction to every such integral curve.
  Hence the assumption $N>0$ yields a contradiction to the
  minimality assumption.
  This proves the claim, i.e., $H/P$ admits an element $v\ne 0$ of degree $1$
  and order $0$.
   
  Next claim: {\em $H$ defines a sub $\O_X$-module sheaf $H_1$}:
  Indeed, if $h$ is a holomorphic function on $X$, then
  $[v,h\Theta]=hv$ for any $v\in L/P$ (note that $\Theta\in P$).
  $H$ is a Lie subalgebra, and $\Theta\in P\subset H$,
  hence $[v,h\Theta]\in H$ for every $v\in H$.

  Thus we have
  \[
  H\ni v=\sum_{j,k}f_{j,k}(x)y_j
  \frac{\partial}{\partial x_k}
  \]
  such that $\exists j,k:f_{j,k}=1$.
  Wlog $f_{1,1}=1$.

  Fix $m$.
  Since $X$ is Stein and the tangent sheaf is coherent, we can find
  a global vector field $w$ on $E$ such that locally
  \[
  w=w_m=\sum_{j,l} h_{j,l}(x) x_j  \frac{\partial}{\partial x_l}
  \]
  with
  \[
  h_{j,l}(p)=\begin{cases} 1 & \text{ if $j=1,l=m$}\\
  0 & \text{ else. }\\
  \end{cases}
  \]
  We lift $w_m$ to a horizontal vector field $\tilde w_m$ on $E$
  (with respect to some holomorphic connection).
  \[
  \tilde w_m=\sum_{j,l}
  h_{j,l}(x) x_j  \left(\frac{\partial}{\partial x_l}
  +\sum_{i,k} A_{i,k,l}(x)y_i\frac{\partial}{\partial y_k}
  \right).
  \]
  We recall: A holomorphic connection on a vector bundle is unique
  only up to
  adding a section in $End(E)\tensor\Omega^1$. Since $X$ is a
  Stein manifold, the vector bundle $End(E)\tensor\Omega^1$
  is spanned by global sections. 
  It follows that by an appropriate choice of the holomorphic
  connection we may assume that $A_{i,k,l}(p)=0$.

  Then
  \[
    [v,\tilde w_m]=\sum_{\beta,r} g_{\beta,r}(x)
    y^ \beta\frac{\partial}{\partial x_r}
    \text{ modulo vertical vector fields and terms vanishing in $p$}
  \]
  with
  \[
  g_{\beta,r}=
  \begin{cases} 1 & \text{ if $\beta=e_1,r=m$}\\
  0 & \text{ else. }\\
  \end{cases}
  \]

  In other words, there is (for every $m$)
  an element $u=[v,\tilde w_m]\in H$
  with
  \[
  u(p)=y_1 \frac{\partial}{\partial y_m}.
  \]
  
  Recall that  $\sV_d$ is a coherent sheaf on the Stein manifold $X$
  (for every $d\in   \N$).
  Fix a multi-index $\beta$.
  Then there exists a vector field $z$ with
  \[
  z=z_{\beta}=
  \sum_{\gamma,l} b_{\gamma,l}(x)
  y^{\gamma}
  \frac{\partial}{\partial y_l}
  \]
  with
  \[
  b_{\gamma,l}(p)=\begin{cases} 1 & \text{ if $\beta=\gamma,l=1$}\\
  0 & \text{ else }\\
  \end{cases}
  \]

  Now
  \[
    [z,u](p)=y^ \beta \frac{\partial}{\partial x_l}
    \]
    Since $\beta$ and $l$ are arbitrary, we see that elements in $H$ span
    the vector bundle corresponding to the sheaf $\sL_d$ in $p$ for
    every $d$. Note that $p$ was also arbitrary. Combined with the fact that
    $H$ is the set of sections in a $\O_X$-module sheaf
    $\sH$ we may deduce that
    $\sH_d=\sL_d$ for all $d$, which implies that $H_d=L_d$.

    Finally, recall that $H$ is assumed to be closed in $L$.
    Therefore $L_k\subset H\ \forall k$ implies $H=L$.
    \end{proof}

\section{Supporting Lemmas}

The present section and the one immediately following it establish a few facts needed in the proof of the line bundle case.

\begin{lemma}\label{lem-x-0}
  Let $V,W ,Y$ be vector fields on a manifold, and $\phi$ a function.

  Assume that $V\phi=0$, $W\phi=0$, $[W,Y]=0$ and $[W,V]=V$.

  Define $\zeta=VY\phi$.

  Then $W\zeta=\zeta$
  and $\zeta W$ is contained in the Lie algebra generated
  by $\phi W$, $Y$, $\phi Y$ and $V$.
\end{lemma}
  
\begin{proof}  
  \begin{align*}
    &\phantom{=}[\phi W,[Y,V]]-[\phi Y,V]\\
    &=-[Y,[V,\phi W]]-[V,[\phi W,Y]]-\phi[Y,V]+\underbrace{(V\phi)}_{=0}Y\\
    &=-[Y,\phi
      \underbrace{[V,W]}_{=-V}]-
    [Y,\underbrace{(V\phi)}_{=0}W]
    - [V,\phi\underbrace{[W,Y]}_{=0}]+[V,(Y\phi)W]-\phi[Y,V] \\
    &=[Y,\phi V]-0
    -0+(VY\phi)W+(Y\phi)\underbrace{[V,W]}_{=-V}-\phi[Y,V]\\
&=
    \colorbox{yellow}{$\phi[Y,V]$}
    +\colorbox{pink}{$(Y\phi)V$}
   +(VY\phi)W-
    \colorbox{pink}{$(Y\phi)V$}
    -\colorbox{yellow}{$\phi[Y,V]$}
    \\
    &=\underbrace{(VY\phi)}_{\zeta}W \\
    &=\zeta W\\
  \end{align*}
  Thus $\zeta W$ is in the Lie algebra generated by
  $\phi W$, $Y$, $\phi Y$ and $V$.

  On the other hand
  \begin{align*}
    W\zeta &= WVY\phi\\
    &= WY\underbrace{V\phi}_{=0}+W[V,Y]\phi\\
    &= 0+[V,Y]\underbrace{W\phi}_{=0}+[W,[V,Y]]\phi\\
    &= -\left([V,\underbrace{[Y,W]}_{=0}]+
      [Y,\underbrace{[W,V]}_{=V}]\right)\phi\quad\text{\sl by Jacobi-Identity}\\
      &=-[Y,V]\phi=-\left(Y\underbrace{V\phi}_{=0}-VY\phi\right)=\zeta
  \end{align*}
\end{proof}

\begin{lemma}\label{lem-x}
  Let $X$ be a smooth affine variety of dimension at least two,
  $V$ a non-zero LND on $X$. Then there exists an algebraic vector field $Y$
  and a regular function $f$ on $X$ such that
  $Vf\equiv 0$ and $VYf\not\equiv 0$.
\end{lemma}

\begin{proof}
  Let $f$ be a  non-constant function which is invariant under
  the flow of $V$ i.e. $Vf\equiv 0$.
  (This is possible, because the flow of $V$ is unipotent and
  $\dim X>1$.)
  Let $q\in X$ be any point such that
  $(df)_q\ne 0$ and $V(q) \ne 0$. Then we may chosen local holomorphic coordinates
  near $q$ such that
  \[
  V=\frac{\partial}{\partial x_1}\text{ and }
    \frac{\partial f}{\partial x_2}(q)= 1
    \]
    Now we choose $Y$ such that near $q$ we have
    \[
    Y=x_1\frac{\partial }{\partial x_2}+
    \left \{\text{terms vanishing with multiplicity $\ge 2$ in $q$}
    \right\}
    \]
    (This choice is possible, because $X$ is affine.)
    Our choice of $Y$ implies $(VYf)(q)=1\ne 0$.
\end{proof}

\begin{proposition}\label{l-vertical}
  Let $\pi:E \to X$ be a holomorphic vector bundle over a flexible smooth
  affine variety, $\dim(X)\ge 2$.
  Let $\Theta$ denote the Euler vector field on $E$.
  Let $k\ge 1$.  Let $D_k$ denote the space of
  vector fields on $E$ of weight $k$
  (i.e.,~$[\Theta,V]=kV$ for $V\in D_k$).
  Let $V_k\subset D_k$ denote
  the subspace of {\em vertical} ones. 

  Let $A$ be a Lie subalgebra of $\Gamma(E,TE)$ which
  contains all vector fields of weight $\le 0$.

  Assume that
  \begin{equation}\label{x-2}
    (A \cap D_k)+V_k=D_k
  \end{equation}

  Then $A\cap V_k\ne\{0\}$.
\end{proposition}

\begin{proof}
  Using lemma~\ref{lem-x}
  we choose vector fields $Y_0$, $V_0$
  and a regular function
  $\phi_0$ on $X$ such that $V_0\phi_0\equiv 0$ and
  $V_0Y_0\phi_0\not\equiv 0$.

  We lift $Y_0$ and $V_0$ to horizontal vector fields $Y$ resp.~$V_1$
  on $E$ using some
  holomorphic connection and define $\phi=\phi_0\circ\pi$.
  Note that $[\Theta,Y]=0=[\Theta,V_1]$
  (Lemma~\ref{euler-hori-commute}) and $\Theta\phi=0$.

  Next we choose a non-zero
  section $\lambda\in\Gamma(X,E^*)$ and observe that
  $\lambda$ defines a fiberwise linear function on $E$.
  We define $V=\lambda^k V_1$.
  From $[\Theta,V_1]=0$ it follows that $[\Theta,V]=(\Theta \lambda^k)V_1=kV$.
  Hence $V\in D_k$.
  Due to \eqref{x-2}, $V-H\in V_k$ for some $H\in A$.

  We observe that $(Y_0\phi_0)\circ\pi=Y(\phi_0\circ\pi)=Y(\phi)$
  (as may be verified easily calculating in local coordinates).

  Similarily it is noted that
  \[
  H(Y\phi)=V(Y\phi)=\left(V_0Y_0\phi_0\right)\circ\pi\ne 0
  \]
  and
  \[
  H(\phi)=V(\phi)=(V_0\phi_0)\circ\pi\equiv 0
  \]

  Now we may conclude from lemma~\ref{lem-x-0}
  that
  \begin{equation}\label{finaleq}
  [\phi \Theta,[Y,H]]-k[\phi Y,H]=(HY\phi)\Theta
  \end{equation}

  We claim that $\phi \Theta,Y,H,\phi Y\in A$.
  $H\in A$ by the choice of $H$.
  Recall that  $[\Theta,Y]=0=\Theta\phi$, implying
  \[
    [\Theta,\phi Y] = \underbrace{(\Theta\phi)}_{=0}Y+\phi\underbrace{[\Theta,Y]}_{=0}=0.
    \quad [\Theta,\phi \Theta]=\underbrace{(\Theta\phi)}_{=0}\Theta+
    \phi\underbrace{[\Theta,\Theta]}_{=0}=0.
  \]
  Thus $\Theta$, $\phi \Theta$, $Y$ and $\phi Y$ are all of weight $0$.
  By assumption, $A$ contains all vector fields of weight $0$.

  Therefore the left hand side of \eqref{finaleq} is contained in $A$.

  On the other hand, by construction the function  $\phi$ is  constant
  along the $\pi$-fibers and $V-H$ is a {\em vertical} vector field.
  Hence $(V-H)\phi=0$ which implies that the right hand side of
  \eqref{finaleq} equals $(VY\phi)\Theta$. 
  This implies the assertion, since $V$ has weight $k$
  (implying that $(VY\phi)\Theta\in V_k$) and $VY\phi\not\equiv 0$.
\end{proof}

\begin{remark}
  The assumption $\dim(X)\ge 2$ in Proposition~\ref{l-vertical}
  is not seriously restrictive for the following reason:
  {\em
If $X$ is a one-dimensional flexible smooth affine
  variety, then $X\simeq\C$ and consequently $E\simeq\C^{N+1}$.
  
}  
\end{remark}

\section{Generalities}

Let $\pi:E\to X$ be a holomorphic vector bundle on a Stein manifold.

We decompose the space $D$ of all vector fields with respect to the weights
of the Euler vector field, i.e., let
\[
D_k=\{v\in D: [\Theta,v]=kv\}
\]

From the Jacobi identity we obtain:
\[
\forall k,m: \forall v\in D_k, w\in D_m: [v,w]\in D_{k+m}
\]

We consider the subspace of {\em vertical} vector fields
\[
V_k=\{ v\in D_k: (d\pi)_p(v)=0\ \forall p\in E\}
\]
and the quotient
\[
Q_k=D_k/V_k
\]

We observe that (locally) the flows corresponding to vector fields
in $D_{-1}\oplus D_0$ preserve the vector bundle structure
and
that those in $D_{-1}\oplus D_0\oplus\bigoplus_k V_k$
are $\pi$-compatible.

Every vector field $v\in V_{-1}\oplus V_0$ is complete,
because the associated flow is by fiber wise affine-linear
transformations.

For $k=-1$ we have
\[
D_{-1}=V_{-1}\simeq\Gamma(X,E)
\]
A vector field $v\in D_{-1}$ may be written in local coordinates
as
\[
\sum_j f_j(x)\dd{y_j}
\]
and corresponds to the section
\[
x\mapsto (x;f_1(x),\ldots,f_r(x))
\]



\begin{lemma}\label{reduce-k}
  Let $A\subset D$ be a vector subspace which is stable
  under $ad(w)$ for all $w\in D_{-1}$.
  Let $k>0$ and assume that $(A\cap D_k)\setminus V_k\ne \emptyset$.

  Then $(A\cap D_j)\setminus V_j \ne \emptyset$ for all $j\le k$.
\end{lemma}

\begin{proof}
  Let $v\in A\cap D_k\setminus V_k$. 
  In local coordinates around some point $p$:
  \[
  v=\sum_{\alpha,j} f_{\alpha,j}(x) y^\alpha\dd{x_j}+\text{vertical terms}
  \]
  Fix $\beta$ such that $\exists j: f_{\beta,j}\ne 0$.
  Observe that elements in $D_{-1}$ are locally written as
  \[
  w=\sum_h g_h(x)\dd{y_h}
  \]
  Now we choose an element $w\in D_{-1}$ with $w(y^ \beta)\not\equiv 0$.
  Then
  \[
    [w,v]=\sum_{\alpha,j}f_{\alpha,j}(x) w\left(y^ \alpha\right)\dd{x_j}
    +\text{ vertical terms }
  \]
  Therefore $[w,v]\not\in V_{k-1}$ and consequently
  \[
  A  \ni [v,w]= D_{k-1}\setminus V_{k-1}
  \]
  Iterating this process yields the assertion.
\end{proof}
  
Now for every function $\phi$ on $X$,
$ad(\phi \Theta):v\mapsto [\phi \Theta,v]$ acts
naturally on $D_k$; on $V_k$ and $Q_k$ it acts by multiplication with $\phi k$:

\[
  [\phi \Theta,v]=\underbrace{(v\phi)}_{=0}\Theta+\phi
  \underbrace{[\Theta,v]}_{=kv}=\phi k v\text{ if $v\in V_k$}
\]

Note that $(v\phi)\Theta \in V_k$ for every $v\in D_k$.
Hence
\[
  [\phi \Theta,v]\equiv \phi kv
  \quad\text{ modulo $V_k$ $\forall v\in D_k$}
\]

As a consequence we obtain:
\begin{proposition}\label{o-module}
  Let $A$ be a vector subspace of $D$ which is invariant under
  $ad(\phi \Theta)$ for every holomorphic function $\phi$ on $X$.

  Then $A\cap V_k$ and $(A\cap D_k)/(A\cap V_k)$ are
  $\Olo_X$-sub modules of $V_k$ resp.~$Q_k$.
\end{proposition}

\begin{lemma}\label{all-vk}
  Assume $rank(E)=1$, i.e., $E$ is a line bundle.
  Assume in addition
  that a group $G$ of holomorphic automorphisms acts on $E$ such that
  $\pi:E\to X$ is equivariant and such that the induced action on the
  base manifold $X$ is transitive.

  Let $A\subset D$ be a closed $G$-invariant
  vector subspace. Assume that $A$ is invariant under
  $ad(\phi \Theta)$ for every holomorphic function $\phi$ on $X$.
  Then (for any $k\in\N$):
  
  {\em If $A\cap V_k\ne\{0\}$, then $V_k\subset A$.}
\end{lemma}

\begin{proof}
 Since $rank(E)=1$, $V_k$ is the space of global sections
 of a line bundle.
 
 Recall that  $V_k\cap A$ is
 a $\mathcal \Olo_X$-module (Proposition~\ref{o-module}).
 Hence, using Theorem~\ref{dense-sections},
 it suffices to show that for every $p\in X$ there exists
 an element $s\in V\cap A_k$ with $s(p)\ne 0$.
 However, this is obvious, because $G$ acts transitively on $X$
 and $A\cap V_k\ne\{0\}$.
\end{proof}
  
\begin{lemma}\label{all-qk}
  Let $E\to X$ be a holomorphic line bundle over a Stein manifold $X$.
  Assume that a Lie group $G$ acts transitively on $X$ such that
  the action can be lifted to an action by vector bundles automorphisms
  on $E$.

  Let $k\in\N$, and let $A$ be a
  closed $G$-invariant Lie subalgebra of $D$ with $D_0\subset A$.
  
  Then:
  
  If $(A\cap D_k)/(A\cap V_k)\ne\{0\}$, then
    $D_k\subset V_k+A$.
\end{lemma}

\begin{proof}
  The elements in $D_k/V_k$ may be regarded as global sections in
  the corresponding vector bundle. This vector bundle is
  isomorphic to $TX \tensor(E^*)^k$.

  Let $v\in (A\cap D_k)\setminus V_k$ and let $p\in X$ such that
  $v(p)$ is a non-zero element in $T_p X\tensor(E_p^ *)^k$.
  
  In local coordinates on a neighorhood $\Omega$ of $p$ ($x_1,\ldots,x_n$ on the base, $y$
  on the fiber):

  \[
  v=v_0y^k,\quad v_0\in\Gamma(\Omega,TX)
  \]
  Let $u\in T_pX$.
  Using lemma~\ref{ex-stein-vf} we may
  choose a vector field $w_0$ on $X$
  with $[w_0,v_0](p)=u$.
  Using a holomorphic connection we may lift $w_0$ to an element
  $\tilde w\in D_0$ with
  \[
  \tilde w=w_0+h(x)y\dd y
  \]
  in local coordinates.
  We choose a global holomorphic function $g\in\Olo(X)$ with
  $g(p)=h(p)$ and define $f(x)=h(x)-g(x)$ and
  \[
  w\stackrel{def}{=}\tilde w-g(x)\Theta=w_0+
  \underbrace{(h(x)-g(x))}_{f(x)}y\dd{y}
  \]
  Note that by construction $f(p)=0$ and $w\in D_0$.
  Now
  \[
    [ w,v]=[w_0+f(x)y\dd y,v_0y^ k]=[w_0,v_0]y^k
    +kf(x)y^kv_0-y^{k+1}(v_0(f))\dd y
  \]
  Hence
  \[
    [ w,v](p) =\underbrace{[w_0,v_0]}_{=u}y^k+
   k \underbrace{f(p)}_{=0}y^kv_0-\underbrace{y^{k+1}(v_0(f))\dd y}_{\in V_k}
  \]  
  Thus $A\cap D_k$ contains an element $r$ with $r(p)\equiv u\ \text{mod}\ V_k$.

  Recall that $u$ was arbitrary. 
  Recall further that $A$ is $G$-invariant and that $G$ acts transitively
  on $X$. It follows that the elements in $(A\cap D_k)/(A\cap V_k)$
  generate the vector bundle $\Omega_X\tensor(E^ *)^k$ everywhere.

  Due to lemma~\ref{o-module},
  $(A\cap D_k)/(A\cap V_k)$ is a $\Olo_X$-sub module of $D_k/V_k$.
  From Theorem~\ref{dense-sections} it follows that $D_k\subset A+V_k$.
  \end{proof}
  
\begin{lemma}\label{go-up}
  If $V_1\subset A$ and
       $(A\cap D_{k})/(A\cap V_{k})\ne\{0\}$,
      then $(A\cap D_{k+1})/(A\cap V_{k+1})\ne\{0\}$.
\end{lemma}

\begin{proof}
  Let $v\in V_1\setminus \{0\}$.
  Then $v=\phi\Theta$ for some non-zero holomorphic section $\phi$
  of $E^*$, where $\Theta$ denotes the Euler vector field.

  Let $w\in (A\cap D_k)\setminus V_k$.
  In local coordinates on some open neighborhood $\Omega$
  of $p$ in $X$:
  \[
  v=g(x)y^2\dd{y}\quad\text{ $\phi=g(x)y$}
  \]
  and
  \[
  w=w_0y^k+f(x)y^{k+1}\dd{y} \quad\text{ with $w_0\in\Gamma(\Omega,T\Omega)$}
  \]
  Then
  \[
    [v,w]=[g(x)y^2\dd{y},w_0y^k+f(x)y^{k+1}\dd{y} ]
    =ky^{k+1}g(x)w_0+\text{ vertical terms}
  \]
%
\end{proof}


\section{Main result for rank at least two}
  
\begin{theorem} \label{mainrank2}
  Let  $X$ be a Stein manifold with the density property, $\pi:E\to X$ a holomorphic
  vector bundle of rank greater than one.

  Assume that the total space $E$ admits an automorphism which is not
  compatible with $\pi$
  (in these sense of Definition~\ref{def-compat}).

  Then $E$ satisfies the density property.
\end{theorem}

\begin{proof}
 Let $L$ be the Lie algebra of all holomorphic vector fields on the manifold
 $E$ and let $P$ denote the subalgebra of those vector fields
 (whose flows are) preserving the projection map $\pi$.

 Vertical vector fields are approximable due to lemma \ref{vertical}.
 Complete vector fields on $X$ may be lifted to complete vector fields
 on $E$ which are compatible with the projection, by using a holomorphic
 connection (Proposition \ref{horiz}).
  
 Combined, we obtain that every vector field in $P$ is approximable.

 Let $H$ denote the Lie subalgebra of $L$ containing all approximable
 vector fields.

 Then
 \[
 P\subsetneq H\subset L
 \]

 Now Theorem \ref{no-inter} implies that $H=L$.
\end{proof}

\section{The line bundle case} \label{LineBundle}

\begin{theorem}\label{mainrank1}
  Let $\pi:L\to X$ be a holomorphic line bundle over a Stein manifold $X$
  with the
  density property.
  Assume that there exists a holomorphic automorphism of $L$ which
  is not $\pi$-compatible
    (in these sense of Definition~\ref{def-compat}).

  Then the total space $L$ is a Stein manifold with the density property,
  too.
\end{theorem}

\begin{proof}
  Since $X$ satisfies the density property, $\dim(X)\ge 2$.

  Let $A$ denote the closure of the vector subspace generated
  by the complete vector fields on $L$
  in the Fr\'echet vector space $D$ of all vector fields on $L$.

  Note that $A$ is a Lie subalgebra of the Lie algebra $D$ due to
   Corollary~\ref{closure-is-LA}.
  
  Recall that $D_{-1}\oplus D_0\subset A$, where $D_k = \{ v \in D : [\Theta, v] = kv \}$ for $k \in \ZZ$ is the set of homogeneous vector fields of degree $k$ with respect to the grading induced by the Euler vector field $\Theta$.

  We assumed that there is a holomorphic automorphism of $L$ which
  is not $\pi$-compatible. Due to Proposition~\ref{from-auto-to-vf}
  it follows that there is a complete vector field $v$ which is not
  $\pi$-compatible.
  The vector subspace of all $\pi$-compatible vector fields
  equals
  \[
  C\stackrel{def}=D_{0}\oplus D_{-1}\oplus\bigoplus_{k>0}V_k
  \]
  The Euler vector field $\Theta$ is an element of $D_0$ and contained
  in the Lie algebra $A$.
  Applying Theorem~\ref{pweyl} to the Fr\'echet space $A/(A\cap C)$ for
  the $\C^ *$-action corresponding to $\Theta$, we conclude
  that
  \[
  \exists k>0: (A\cap D_k)/(A\cap C\cap D_k)\ne\{0\}
  \ \iff\ \exists k>0: (A\cap D_k)\setminus V_k\ne  \emptyset
  \]
  With lemma~\ref{reduce-k} this in turn implies
  \[
  (A\cap D_1)\setminus V_1\ne \emptyset
  \]
  Next $D_1\subset V_1+A$ due to lemma~\ref{all-qk}.
  By Proposition~\ref{l-vertical}, $A\cap V_1\ne\{0\}$.
  By lemma~\ref{all-vk}, $V_1\subset A$. In combination
  with $D_1\subset V_1+A$, this yields $D_1\subset A$.

  We have established that $V_1\subset A$ and
  $D_1\subset A$. Repeated application of lemma~\ref{go-up}
  implies that
  \[
  (A\cap D_k)\setminus V_k\ne \emptyset\ \forall k\in\N
  \]
  This in turn
  (using once again lemma~\ref{all-qk}, Proposition~\ref{l-vertical}
  and lemma~\ref{all-vk}) yields $D_k\subset A\ \forall k$.
  Since $A$ is closed and $\oplus_k D_k$ is dense in $D$
  (Theorem~\ref{pweyl}), we obtain $A=D$.
  \end{proof}

\section{Flexibility} \label{flexibility}

This Section aims to prove that the total space of an algebraic vector bundle over a flexible affine variety is flexible, too.

\begin{theorem}\label{bass-quillen}
  Let $E\to B$ be an algebraic vector bundle 
  over a smooth affine
  $k$-variety $B$.

  Let $v$ be a locally nilpotent derivation of $k[B]$.
  
  Then the vector bundle $E$ is invariant under the associated flow,
  i.e.,
  $\mu_t^*E\simeq E$ 
  where $\mu_t$ denotes the automorphism of $B$ corresponding to the
  $k$-algebra automorphism
  \[
  f\mapsto \sum_{n=0}^\infty \frac{t^nv^n(f)}{n!}
  \]
\end{theorem}

\begin{proof}
  In \cite{MR641133} (see also \cite{MR2235330})
  Lindel proved (partially) the Bass-Quillen conjecture.
  Lindel's result implies that, given a smooth affine variety $V$ and $d\in\N$,
  every
  vector bundle on $\A^d\times V$ is given as a pull-back of a vector bundle
  on $V$ via the projection map $\A^d\times V\to V$.

  Let $\mu:G_a\times  B\to B$ be the map defining the group action
  induced by $v$, i.e., $\mu(t,x)=\mu_t(x)$.
  and let $p:G_a\times  B\to B$ be the projection on the second factor.
  By Lindel's result there is a vector bundle $E_0$ on $B$ such that
  $\mu^*E  \simeq p^ *E_0$. 
  We observe that $\mu=p$ on $\{0\}\times B$.
  Hence $E_0\simeq E$.
  It follows that
  \[
  \mu_t^*E\simeq\mu^*E|_{\{t\}\times B}\simeq
  p^*E_0|_{\{t\}\times B}\simeq 
  p^*E|_{\{t\}\times B}\simeq E.
  \]
  \end{proof}

\begin{lemma}\label{split}
  Let
  \begin{equation}\label{uni-seq}
    1 \to N \to G \to U\to 1
  \end{equation}
  be a short exact sequence of linear algebraic groups
  with $U\simeq G_a$.

  Then the sequence splits.
\end{lemma}

\begin{proof}
  Let $S_1$ be a maximal connected reductive subgroup of $N$ and let
  $S$ be a maximal connected reductive subgroup of $G$ with $S_1\subset  S$.
  Because $N$ is normal in $G$, the intersection $S\cap N$
    is normal in $S$ and therefore reductive, too.
    Maximality of $S_1$ implies that $S_1$ equals the connected component
    of the neutral element of $S\cap N$, which we denote as
    $(S\cap N)^ 0$. Hence $S_1$ is a normal subgroup of $S$.
    
  Now $S/S_1$ is a reductive group of dimension at most one which admits
  a non-trivial morphism of algebraic groups into the additive group $G_a$.
  Thus $S/S_1$ is trivial, i.e., $S_1=S$.
  Let $V_1$ resp.~$V$ denote the unipotent radical
  (i.e. the maximal normal unipotent subgroup)
  of $N$ resp.~$G$.
  Since $\dim G/N=1$ and $S=S_1$, we have $\dim(V/V_1)=1$.
  In particular $V\not\subset N$. We choose a one-dimensional vector
  subspace $L$ of $Lie(V)$ such that
  $L\not\subset Lie(N)$.
  Now $\exp(L)$ is a one-dimensional unipotent subgroup of $G$ which admits
  a non-trivial morphism of algebraic groups $\exp(L)\to U$.
  Since $\exp(L)$ and $U$ are both one-dimensional and unipotent,
  this morphism is an isomorphism. Therefore it admit an inverse
  which gives us the desired splitting
  of \eqref{uni-seq}.
\end{proof}

\begin{proposition}\label{ga-lift}
  Let $E\to B$ be an algebraic vector bundle over a smooth affine
  variety $B$.

  Then every action of the additive group $G_a$ on $B$ lifts to an
  action on $E$.
\end{proposition}

\begin{proof}
  For each $t\in G_a$, let $\mu(t)$ denote the corresponding
  automorphism of $B$. Since $\mu(t)^*E\simeq E$ (Theorem~\ref{bass-quillen})
  we may lift $\mu(t)$ to an automorphism of $E$.

  Let $G$ denote the group of all vector bundle automorphisms of $E$
  such that the respective automorphism of the base manifold $B$ is
  contained in $\mu(G_a)$.
  In other words $G$ consists of all variety automorphisms $\phi$ of $E$
  which are fiberwise linear and such that there exists a $t$ with
  $\pi(\phi(x))=\mu(t)(\pi(x))$ for all $x\in E$.

  Now there is a natural surjective morphism of algebraic groups
  from $G$ to $G_a$ and lemma~\ref{split} implies that the $G_a$-action
  on $B$ lifts to an action on $E$ via an embedding of $G_a$ into $G$.
\end{proof}

\begin{theorem}\label{vb-flexible}
  Let $E\to B$ be an algebraic vector bundle. Assume that $B$ is a flexible
  smooth affine variety.

  Then $E$ is flexible, too.
\end{theorem}

\begin{proof}
  For $p\in E$ let $U_p$ denote the set of all tangent vectors $v\in T_pE$
  for which there exists a corresponding $G_a$-action.
  Recall that $p$ is a {\em flexible point}  if $U_p=T_pE$ and
  that $E$ is flexible if every $p\in E$ is a flexible point.

  Since $B$ is flexible, Proposition~\ref{ga-lift} implies that
  $\pi_*:U_p\to T_{\pi(p)}B$ is surjective for every $p\in E$.

  Next we employ the work of \cite{GaifullinShafarevic2019}.
  We consider the natural $\C^*$-action on the vector bundle $E$.
  In the language used in \cite{GaifullinShafarevic2019}  $E$ is a
  {\em parabolic non-hyperbolic} $G_m$-space.
  Using Proposition 3 in \S3 of \cite{GaifullinShafarevic2019}  we may
  conclude that every point $p$  in the zero-section of $\pi:E\to B$
  is a flexible point in $E$.

  Now we observe that the set of flexible points is open and invariant
  under all algebraic automorphisms of $E$.  The automorphisms of $E$ include
  fiberwise multiplication with scalars $\lambda\in \C^*$.
  Evidently the only open neighborhood of the zero-section which is
  invariant under this $\C^*$-action is $E$ itself.

  Thus $E$ is flexible at every point, i.e.,
  $E$ is a flexible variety.
\end{proof}

The following is an additional result about existence of vector fields.

\begin{proposition}
  Let $X$ be a smooth affine variety with a non-trivial action of
  $U=(\C,+)$. Let $\Theta$ denote the corresponding vector field.

  Then there exists a $U$-invariant  Zariski open subset $\Omega\subset X$
  such that for every $p\in\Omega$ and $v\in T_pX$ there is a vector field
  $W$ on $X$ with $[W,\Theta]=0$ and $W_p=v$.
\end{proposition}

\begin{proof}
  By a result of Rosenlicht (see \cite{PopovVinberg}, Theorem~2.3)
  there is a $U$-invariant Zariski open subset
  $\Omega_0\subset X$ with $\Omega_0\simeq \C\times M_0$ such that the
  $U$-action on $\Omega_0$ is simply translation in the first factor.
  Let $I\subset \C[X]$ be the ideal of functions vanishing on
  $X\setminus\Omega_0$. Now $U$ is unipotent and
  acts linearly on $\C[X]$, stabilizing $I$.
  Therefore there is a non-zero $U$-invariant element $f$
  in $I$. Define $\Omega=\{x\in\Omega_0:f(x)\ne 0\}$.
  By construction, $\Omega$ is an open invariant subset of $\Omega_0$.
  Therefore $\Omega\simeq\C\times M$
  for some Zariski open subset $M\subset M_0$.
  Choose $p\in \Omega$, $v\in T_pX$.
  Then $v=c(\Theta_p)+w$ with $w\in T_pM$.
  We choose a vector field $W_0$ on $M$ with $(W_0)_p=w$.
  Now $W_1=c\Theta+W_0$ is a vector field on $\Omega$ with
  $(W_1)_p=v$ and $[W_1,\Theta]=0$. It extends to a rational
  vector field on $X$. Recall that $f$ is a $U$-invariant function which
  vanishes outside $\Omega$. Therefore $W=f^NW_1$ for $N$ sufficiently large
  is a regular vector field with $[W,\Theta]=0$ and $W_p=v$.
\end{proof}

\section{Tame sets}
The notion of {\em tame set} was first introduced by Rosay and Rudin \cite{RosayRudin1988} as they were studying the properties of the group of holomorphic automorphisms of complex Euclidean spaces.

\begin{definition}[\cite{RosayRudin1988},{Def.~3.3}]
\label{def-tame-RR}
Let $e_1$ be the first standard basis vector of $\CC^n$.
A set $A \subset \CC^n$ is called \emph{tame}
(in the sense of Rosay and Rudin) if there exists an $F \in \aut(\CC^n)$ such that $F(A) = \NN \cdot e_1$.
It is very tame if we can choose $F \in \aut_1(\CC^n)$ to be volume preserving.
\end{definition}

In recent years, this notion was generalized to arbitrary complex manifolds. The definition of {\em strongly tame set} is due to Andrist and the first named author \cite{AndristUgolini2019}; {\em weakly tame sets} were introduced by the second named author \cite{Winkelmann2019}.

\begin{definition}[Andrist, Ugolini]
Let $X$ be a complex manifold and let $G \subseteq \aut(X)$ be a subgroup of its group of holomorphic automorphisms. We call a closed discrete infinite set $A \subset X$ a \emph{strongly $G$-tame set} if for every injective mapping $f \colon A \to A$ there exists a holomorphic automorphism $F \in G$ such that $F|_{A} = f$. If $G = \aut(X)$, we only say {\em strongly tame set}.
\end{definition}

\begin{definition}[Winkelmann]
Let $X$ be a complex manifold.
An infinite discrete
subset $D$ is called {\em weakly tame}
if for every exhaustion function $\rho \colon X \to \RR$
and every function $\zeta \colon D \to \RR$
there exists an automorphism $\Phi$ of $X$
such that $\rho(\Phi(x)) \geq \zeta(x)$ for all $x \in D$.
\end{definition}

Any strongly tame set is automatically weakly tame, hence their names. When $X = \C^n, \ n >1$, both notions coincide with the classical definition by Rosay and Rudin.

It is clear that the existence of a tame set (strong or weak) requires the group of holomorphic automorphisms to be somewhat large. In this section, we provide existence results for total spaces of holomorphic vector bundles.

\begin{proposition}
Let $E \to B$ be a holomorphic vector bundle over a Stein manifold $B$. Then $E$ admits a weakly tame set.
\end{proposition}

See \cite{Winkelmann2019}, Proposition 8.4 for a closely related result.

\begin{proof}
   Choose a closed discrete sequence $A = \{a_n\}_{n \in \N} \subset E$ such that $\pi(a_n) \neq \pi(a_m)$ for $n \neq m$,  $a_n \neq 0 \in E_{\pi(a_n)}$, and such that $\pi(A)$ is also closed and discrete in $B$.
  
  Fix an exhaustion function $\rho$. Given a function $\zeta \colon A \to \RR$, we can find a holomorphic function $h$ on $B$ such that $\rho (e^{h(\pi(a_n))} \cdot a_n) > \zeta(a_n)$, because $A$ is a discrete subset in the
  Stein manifold $B$.

Let $\Theta$ be the Euler vector field, then the time-1 flow $\Phi$ of the vector field $h\Theta$ is such that $\rho(\Phi(x)) \geq \zeta(x)$ for all $x \in A$ i.e. $A$ is weakly tame.
\end{proof}

As we saw, the existence of a weakly tame set in the total space of a vector bundle over a Stein manifold does not require any additional hypothesis. The picture is quite different for strongly tame sets.

Given a complex manifold $X$, we denote by $\aut(X)^0$ the connected component of the identity in $\aut(X)$.

\begin{proposition}
Let $E \to B$ be a holomorphic vector bundle over a Stein manifold $B$ and assume that $B$ admits a strongly $\aut(B)^0$-tame set. Then, $E$ admits a strongly tame set.
\end{proposition}

\begin{proof}
Let $A'  \subset B$ be a strongly $\aut(B)^0$-tame set. Denote by $A = \{a_n\}_{n \in \N}\subset E$ the image of $A'$ under the zero section of $E \to B$.

Let $f \colon A \to A$ be an injective self-map. Since $A'$ is $\aut(B)^0$-tame, there exists an automorphism $\Phi \in \aut(B)^0$ such that $\Phi(\pi (a_n)) = \pi (f(a_n))$. The map $\Phi$ induces a commutative diagram

\[ \begin{tikzcd}
\Phi^\ast E \arrow{r}{\tilde{\Phi}} \arrow[swap]{d}{} & E \arrow{d}{} \\%
B \arrow{r}{\Phi}& B
\end{tikzcd}
\]

Since the map $\Phi$ is connected to the identity, $E$ and $\Phi^\ast E$ are topologically isomorphic as vector bundles over $B$. As the topological and holomorphic classification of vector bundles over Stein manifolds is equivalent, there exists a holomorphic vector bundle isomorphism $F \colon E \to  \Phi^\ast E$ (covering the identity).

The map $H = \tilde{\Phi} \circ F$ satisfies $H(a_n) = f(a_n)$, hence $A$ is strongly tame.

\end{proof}

It is not clear whether reducing ourselves to the connected component of the identity is particularly restrictive. When the base space $B$ is an affine homogeneous space of a linear complex Lie group, the existence of strongly $\aut(B)^0$-tame sets is proved in \cite{AndristUgolini2022X}.

\bibliography{Biblio}
\bibliographystyle{alpha} 

\end{document}